\documentclass[12pt]{amsart}
\usepackage[utf8]{inputenc}

\usepackage[english]{babel}
\usepackage{tabularx}
\usepackage{multirow}
\usepackage{amsmath, amssymb, amsthm, amscd,color,comment}
\usepackage{cancel}
\usepackage{cite}
\usepackage{alltt}
\usepackage[dvipsnames]{xcolor}
\usepackage{array}
\usepackage{caption} 
\usepackage{mathtools}
\usepackage{comment}

\usepackage{enumerate}
\usepackage{cancel}
\newtheorem{theorem}{Theorem}[section]

\newtheorem{remark}[theorem]{Remark}
\newtheorem{lemma}[theorem]{Lemma}
\newtheorem{corollary}[theorem]{Corollary}
\newtheorem{proposition}[theorem]{Proposition}

\newtheorem{conjecture}[theorem]{Conjecture}
\DeclarePairedDelimiter\ceil{\lceil}{\rceil}
\DeclarePairedDelimiter\floor{\lfloor}{\rfloor}

\def\N{\mathbb N}
\def\Z{\mathbb Z}
\def\cC{\mathcal C}

\def\cF{\mathcal F}
\def\cH{\mathcal H}

\def\cX{\mathcal X}
\def\cY{\mathcal Y}

\def\a{\alpha}

\def\fqs{\mathbb F_{q^2}}

\def\fq{\mathbb F_q}

\def\aut{{\rm Aut}}

\def\supp{{\rm Supp}}

\def\dim{{\rm dim}}

\def\a{\alpha}

\newcommand{\al}{\alpha}

\newcommand{\la}{\lambda}
\newcommand{\tx}{\tilde x}
\newcommand{\ty}{\tilde y}

\begin{document}

\title{Some families of non-isomorphic maximal function fields}



\author{Peter Beelen, Maria Montanucci, Jonathan Tilling Niemann and Luciane Quoos}

\begin{abstract}
The problem of understanding whether two given function fields are isomorphic is well-known to be difficult, particularly when the aim is to prove that an isomorphism does not exist. 
In this paper we investigate a family of maximal function fields that arise as Galois subfields of the Hermitian function field.
We compute the automorphism group, the Weierstrass semigroup at some special rational places and the isomorphism classes of such function fields. 
In this way, we show that often these function fields provide in fact examples of maximal function fields with the same genus, the same automorphism group, but that are not isomorphic.
\end{abstract}
\maketitle

\section{Introduction}

A function field $\mathcal{F}$ defined over a finite field with square cardinality $\fqs$ is called maximal, if the Hasse–Weil bound is attained. More precisely, if $\mathcal{F}$ has genus $g$, the Hasse–Weil bound states that:
$$N(\mathcal{F}) \leq q^2+1+2gq,$$
where $N(\mathcal{F})$ denotes the number of places of degree one of $\mathcal{F}/\fqs$. A function field is called maximal over $\fqs$ if the above bound is attained with equality, that is, $N(\mathcal{F})=q^2+1+2gq$.

An important and well-studied example of a $\fqs$-maximal function field is the Hermitian function field $\mathcal{H}:=\fqs(x,y)$ with $$y^{q+1}=x^q+x.$$ 
The function field $\mathcal{H}$ has genus $g_1=q(q-1)/2$, which is in fact the largest possible genus for a maximal function field over $\fqs$, see \cite{I1982}. This function field has also an exceptionally large automorphism group, which is isomorphic to $\textrm{PGU}(3,q)$. 

In fact more is known, namely the Hermitian function field $\cH$ is the only maximal function field (up to isomorphism) with $g_1$, see \cite{RS1994}. 

Since a subfield of a maximal function field with the same field of constants is maximal by a theorem of Serre \cite{Serre}, computing fixed fields of $\mathcal{H}$ of a subgroup of $\textrm{PGU}(3, q)$ have given rise to many examples of maximal function fields.

 The second largest genus for a maximal function field over $\fqs$ is $g_2=\lfloor \frac{(q-1)^2}{4} \rfloor$ and examples of such function field can be obtained in fact as a subfield of $\mathcal{H}$, namely
$$y^{\frac{q+1}2}=x^q+x,  \quad \text{ for $q$ odd (see \cite{FGT1997})}, $$
and 
$$y^{q+1}=x^{q/2}+x^{q/4}+\dots +x^2+x, \quad \text{ for $q$ even (see \cite{AT1999})}. $$

In \cite{FGT1997} and \cite{AT1999} it is proven that the two function fields above are the only maximal function fields, up to isomorphism, of genus $g_2$. 

Inspired by this, it is natural to ask how many maximal function fields of a given genus one can find over $\fqs$. 
This question is already unanswered for $g_3=\lfloor (q^2-q)/6 \rfloor$, where in fact uniqueness is still  an open problem.

The underlying reason for this is that it is in general quite difficult to decide whether two given maximal function fields over the same finite field are isomorphic or not. Besides the function fields with large genus mentioned above there are in fact very few examples in the literature where the investigation of non-isomorphic maximal function fields is carried out completely and successfully.  

In \cite[Theorem 1.1]{GHKT2006}, the authors classify when ${\frac{q+1}2} > 3$ is a prime number and $ 1 \leq i \leq \frac{q-3}2$, the family of maximal function fields $\fqs(x,y)$ of genus $q-1$ defined by
$$y^{q+1}=x^{2i}(x^2+1).$$

This family is a collection of fixed fields of the Hermitian function field with respect of cyclic automorphism groups of order prime to $p$, the characteristic of $\fqs$. In the same paper \cite{GHKT2006} the authors also prove in fact that the family described above gives rise to roughly $(q+1)/12$ non-isomorphic maximal function fields.

Another example of this type was given in \cite{ABQ2009}, where the authors investigated isomorphism in the family of maximal function fields $\fqs(x,y)$ with $y^k=-x^b(x^d+1)$ where $kd= q^n+1$ and $b \in \N$ satisfying $\gcd(k, b)=\gcd(k, b+d)=1$ of genus $\frac{k-1}{2}$.

In this paper we present a large class of non-isomorphic maximal function fields over $\fqs$ of same genus $(q-1)/2$, $q$ odd. This function fields arise as a generalization of a family of maximal function fields that appeared in \cite{GMQ2021}, where the authors provide $6$ non-isomorphic maximal function fields of genus $12$ over $\mathbf{F}_{5^4}$. More precisely for $q$ odd, we consider the function field $\mathcal{F}_i(x,y)$ with
$$y^m=x^i(x^2+1),$$
where $i \in \mathbb{Z}$, $m=(q+1)/2$ and $\gcd(i,m)=\gcd(i+2,m)=1$.


As in \cite[Theorem 1.1]{GHKT2006} this family is a collection of subfields of the Hermitian function fields, see \cite[Example 6.4]{GSX2000} and hence it is $\fqs$-maximal.

In this paper, we compute the automorphism group of $\mathcal{F}_i$ over $\bar{\mathbb{F}}_{q^2}$, the Weierstrass semigroup at some places of degree one and the isomorphism classes in the family $\{\mathcal{F}_i\}_i$. 

In this way, we show that these function fields provide in fact examples of maximal function fields with the same genus, often the same automorphism group, but that are not isomorphic.

Finally in Theorem \ref{numbernoniso} we show that the number of isomorphism classes among the function fields $\mathcal{F}_i$ is
$$
        N(m) = 
        \begin{cases}
            \frac{\phi_2(m) + 1}{2} & \text{ for } m \not\equiv 0 \pmod 4,\\
            \frac{\phi_2(m) + 2}{2} & \text{ for } m \equiv 0 \pmod 4,
        \end{cases}
    $$
    where $\phi_2(m)$ is given by 
    $$
        \phi_2(m) = 2^{\max\{0,\alpha_0-1\}} p_1^{\alpha_1 - 1}(p_1-2) \cdots p_n^{\alpha_n -1}(p_n - 2),
    $$
for $m = 2^{\alpha_0} p_1^{\alpha_1} p_2^{\alpha_2} \cdots p_n^{\alpha_n}$,  $p_1, \dots, p_n$ pairwise distinct odd primes and with $\alpha_0 \in \mathbb{Z}_{\geq 0}$ and $\alpha_1, \dots, \alpha_n$ in $\mathbb{Z}_{>0}$.

The paper is organized as follows. In Section 2 some preliminary results on maximal function fields and their automorphism group are provided. In the same section we also describe the inspiration behind the family of function fields $\{\mathcal{F}_i\}_i$, and why they give intuitively rise to many non-isomorphic maximal function fields. In Section 3 the family $\{\mathcal{F}_i\}_i$ is defined and some first isomorphisms between the function fields $\mathcal{F}_i$ are computed, together with partial information about the Weierstrass semigroup at some special places of degree one. Section 4 provides the complete determination of the automorphism group of $\mathcal{F}_i$ over $\bar{\mathbb{F}}_{q^2}$ for $i \ne (m-2)/2$. Finally, in Section 5 the isomorphism classes among the function fields $\{\mathcal{F}_i\}_i$ and its number is computed.

\section{Preliminaries and Notations}
In this section, we deal with the preliminary notions and results that will be needed throughout the paper. In the first subsection, we collect some general properties of maximal function fields, Weierstrass semigroups and automorphism groups that we will use in the later sections. In the second subsection, we recall the definition of the function fields $\{\mathcal{F}_i\}_i$ and we
focus on some particular rational functions defined on it, computing their principal divisors. Some preliminary results about $\aut(\mathcal{F}_i)$ will also be provided there.

\subsection{Automorphism groups and Weierstrass semigroups on algebraic function fields}

Throughout this paper $p$ is a prime and $q=p^h$ where $h \geq 1$. We denote with $K=\bar{\mathbb{F}}_{q}$ the algebraic closure of the finite field $\mathbb{F}_q$, and for a function field $\mathcal{F}$ of genus $g$ defined over the finite field $\mathbb{F}_q$, we denote with $\aut(\mathcal{F})$ the automorphism group of the function field $\mathcal{F}$ over $K$. Together with the genus $g$, the group $\aut(\mathcal{F})$ is a well-known invariant  of $\mathcal{F}$ under isomorphism. In fact if $\mathcal{F}$ and $\mathcal{G}$ are two isomorphic function field over $K$, and $\varphi: \mathcal{F} \mapsto \mathcal{G}$ is an isomorphism, then $\aut(\mathcal{F})=\varphi^{-1} \circ \aut(\mathcal{G}) \circ \phi$. This observation implies also that the set of fixed places of $\aut(\mathcal{F})$ is mapped to the set of fixed places of $\aut(\mathcal{G})$.

If $\mathcal{F}$ is maximal over $\fqs$ then $\aut(\mathcal{F})$ is extremely structured. A first important fact is that actually  $\aut(\mathcal{F})$ is defined over $\fqs$, see \cite[Theorem 3.10]{GSY2015}.
This implies not only that $\aut(\mathcal{F})$ acts on the set of places of degree one of $\mathcal{F}$ (that is, the $\fqs$-rational places), giving important restriction on the shape of $\aut(\mathcal{F})$, but also that $\aut(\mathcal{F})$ can be used to decide whether another given function field $\mathcal{G}$ can be isomorphic to $\mathcal{F}$ or not. 

Another important tool one can use in this direction is the Weierstrass semigroup at places of $\mathcal{F}$. If $P$ denotes a place of $\mathcal{F}$, then the Weierstrass semigroup $H(P)$ at $P$ is defined as the collection of all integers $n \in \mathbb{Z}$ for which there exists a function $\alpha \in \mathcal{F}$ such that $(\alpha)_\infty=nP$.

Clearly $H(P) \subseteq \mathbb{Z}_{\geq 0}$ is a numerical semigroup, and the elements of $G(P):=\mathbb{Z}_{\geq 0} \setminus H(P)$ are called gaps. Since it is well-known that when $\varphi: \mathcal{F} \rightarrow \mathcal{G}$ is an isomorphism then $H(P)=H(\varphi(P))$ for all $P$ place in $\mathcal{F}$, Weierstrass semigroups also represent a useful tool when deciding whether two function fields can be isomorphic or not. 

To compute $G(P)$, and hence $H(P)$, the following result becomes extremely useful and in fact, we will use it in the later sections.

\begin{proposition}\cite[Corollary 14.2.5]{S2006}\label{propgaps}
Let $\cC$ be an algebraic curve of genus $g$ defined
over $K$. Let $P$ be a point of $\cC$ and $\omega $ a regular differential on $\cC$. Then $v_P(\omega) + 1$ is a gap at $P$.
\end{proposition}

As we will see later, both the automorphism group and its fixed places, as well as Weierstrass semigroups will be used to construct our large family of non-isomorphic maximal function fields.

We are now in a position to define the main protagonist of this paper, namely the family of maximal function fields $\{\mathcal{F}_i\}_i$. The aim of the next subsection is to show first where the inspiration behind its construction.

\subsection{The inspiration behind the construction of the function field $\mathcal{F}_i$}

Using reciprocal polynomials, \cite{GMQ2021} provides a family of algebraic function fields $\mathbb{F}_{q^2}(u,v)$ with
$$v^{\frac{q+1}{2}}=\frac{(u+b)(bu+1)}{u^s},$$ 
where $b\in \fq^*$ and $s\geq 0$ is a integer. 

In the particular case $q=5^2$ and $b=2$, this family provides six non-isomorphic maximal function fields $\mathbb{F}_{5^4}$ of genus $12$, more specifically, the function fields $\mathcal{G}_i(u,v)$ with
$$
 v^{13}=\frac{(u+2)(2u+1)}{u^
i}, \text{ for } i=1,3,4, 5,6 \text{ and } 7.
$$
Since $2$ is a $13$-power in $\mathbb{F}_{5^4}^*$, these function fields are isomorphic to the family
\begin{equation}\label{curvefamily1}
y^{13}=x^{i}(x^2+1) \text{ for } i=6,7,8,9, 10\text{ and } 12.
\end{equation}
This family of function fields is maximal since they appear as special cases of \cite[Example 6.4 (Case 1)]{GSX2000} and are subfields of the Hermitian function field $z^{5^2+1}=w^{5^2+1}+1$. This family represents the inspiration for the work of this paper, in fact for $i,j=6,7,8,9,10,12$, with $i \ne j$ the function fields $\mathcal{G}_i$ and $\mathcal{G}_j$ are not isomorphic, even though they have the same genus.
In reality, the function fields $\mathcal{G}_i$ and $\mathcal{G}_j$ have also the same automorphism group. 
In fact $\aut(\mathcal{G}_i)$ contains the cyclic group of order $26$ given by 
$$G_i:=\{\sigma: (x, y)\mapsto (ax, by)\mid a^2=b^{13}=1 \}\subseteq \aut(\mathcal{G}_i) \quad \text{for $i=6, 8, 10, 12$}$$ 
and 
$$G_i:=\{\sigma: (x, y)\mapsto (ax, by)\mid a^2=1, b^{13}=a \}\subseteq \aut(\mathcal{G}_i)\quad \text{for $i=7, 9$}.$$ 
This means that $\mathcal{G}_i$ is a $(g+1)$-function field (equivalently it gives rise to a $(g+1)-curve)$, that is, $g+1$ is a prime number and is a divisor of $|\aut(\mathcal{G}_i)|$. From \cite[Theorem 17]{AS2021}, we have that 
$|\aut(\mathcal{G}_i)|=k(g+1)$, where  $k=2$ or $k=4$. From \cite[Proposition 7]{AS2021}, we conclude that $\aut (\mathcal{G}_i)=G_i$ for all $i=6,7,8,9,10,12$.

It is in fact Weierstrass semigroups that allow us to see that these function fields cannot be isomorphic for different choices of $i$. Denoting with $P_0$ (resp. $P_\infty$) the only zero (resp. only pole) of $x$ in $\mathcal{G}_i$ and with $P_\alpha$ and $P_{-\alpha}$ the zeroes of $x^2+1$ where $\alpha^2+1=0$, one has that the automorphism group $G_i$ described above fixes $P_0$ and $P_\infty$ and has no other fixed places globally. Now if we have an isomorphism $\varphi_{i,j}$ between $\mathcal{G}_i$ and $\mathcal{G}_j$, then the set of fixed points of the automorphism group $G_i$ needs to be mapped to the corresponding one of $G_j$. 
The following tables collects all the  gap sequences obtained using \cite[Propositions 4.1, 4.2]{CMQ2023} at the points $P_0$ and $P_\infty$ for the values of $i$ we are interested in.

\vspace{0.2cm}
\begin{tabular}{ |c|c|c| } 
 \hline
& $i=6$ & $i=7$ \\ 
\hline
$G(P_\infty)$ & $\{ 1, 2, 3, 4, 5, 6, 7, 9, 10, 12, 15, 18\}$ & $\{1, 2, 3, 4, 6, 7, 8, 11, 12, 16, 17, 21\}$ \\ 
\hline
 $G(P_0)$ & $\{1, 2, 3, 4, 5, 6, 9, 10, 11, 12, 18, 19\}$ & $\{1, 2, 3, 4, 5, 7, 8, 9, 10, 14, 15, 20 \}$  \\ 
 \hline
\end{tabular}

\vspace{0.2cm}

\begin{tabular}{ |c|c|c| } 
 \hline
& $i=8$ & $i=9$ \\ 
\hline
$G(P_\infty)$ & $\{ 1, 2, 3, 4, 5, 6, 8, 9, 12, 15, 16, 19 \}$ & $\{ 1, 2, 3, 4, 5, 6, 8, 10, 12, 15, 17, 19  \}$ \\ 
\hline
 $G(P_0)$ & $\{1, 2, 3, 4, 6, 7, 8, 9, 11, 14, 16, 21\}$ & $\{1, 2, 3, 5, 6, 7, 9, 10, 11, 14, 18, 22 \}$ \\ 
 \hline
\end{tabular}

\vspace{0.2cm}

\begin{tabular}{ |c|c|c| } 
 \hline
&  $i=10$ & $i=12$\\ 
\hline
$G(P_\infty)$ &  $\{1, 2, 3, 4, 5, 6, 7, 8, 14, 15, 16, 17 \}$ & $\{  1, 2, 3, 4, 5, 6, 7, 8, 9, 10, 11, 12 \}$\\ 
\hline
 $G(P_0)$ & $\{1, 2, 4, 5, 7, 8, 10, 11, 14, 17, 20, 23 \}$ & $\{ 1, 2, 3, 4, 5, 6, 7, 8, 9, 10, 11, 12\}$ \\ 
 \hline
\end{tabular}
\vspace{0.2cm}

Since according to the three table above there is not $2$ equal gap sequences for distinct values of $i$, we deduce that the curves $\mathcal{C}_i$ with $i=6,7,8,9,10,12$ are pairwise not isomorphic.

The aim of this paper is to consider a generalization of this example, and inspired by the considerations above to show that they give rise to a large class of non-isomorphic maximal function fields with the same genus and often the same automorphism group. This family of maximal function fields is described in the next section, where some initial properties will be proven as well.

\section{The function fields $\cF_i$}\label{sec_initial_results} 

Let $q=p^h$, where $p$ is an odd prime and let $m=\frac{q+1}{2}$. The family of maximal function fields $\cF_i:=\fqs(x,y)$ is defined by
\begin{equation}\label{curveC}
y^m=x^i(x^2+1),\, i \in \mathbb{Z}, \ \gcd(i,m)=\gcd(i+2,m)=1
\end{equation}
over $\fqs$. The condition $\gcd(i,m)=\gcd(i+2,m)=1$ implies that all these function fields have genus $m-1$. In fact, these function fields already appeared in \cite[Example 6.4 (Case 1)]{GSX2000} and are subfields of the Hermitian function field $\fqs(u,v)$ over $\fqs$ defined by the equation $v^{q+1}=u^{q+1}+1$. 

For $\a \in \fqs$ such that $\a^2+1=0$ we denote by $P_{\a}, P_{-\a}$ the only places of $F_i$ associated to $x=\a, x=-\a$ respectively. We denote by $P_0$ (resp. $P_\infty$) the unique zero (resp. pole) in $F_i$ of the function $x$. By considering the Kummer extension $\fqs(x,y)/\fqs(x)$, one directly derives the following divisors, which will be useful later on
\begin{align}\label{divisors}
\begin{split}
&(x)=m(P_0-P_\infty),\\
&(y)=iP_0+P_\alpha + P_{-\alpha} -(i+2)P_\infty,\\
&(dx)=-(m+1)P_\infty + (m-1)(P_0+P_\alpha+P_{-\alpha}).
\end{split}
\end{align}


Our aim is to study the isomorphism classes of the function fields $\cF_i$, with $i \in \mathbb{Z}$ satisfying $\gcd(i,m)=\gcd(i+2,m)=1$. As a matter of fact, we will give a complete description of the isomorphism classes of the function fields $\overline{\mathbb{F}}_{q^2}\cF_i$, where $\overline{\mathbb{F}}_{q^2}\cF_i$ denotes the function field obtained from $\cF_i$ by extending the constant field to $\overline{\mathbb{F}}_{q^2}$, the algebraic closure of $\fqs.$

We start with a lemma that will allow us to restrict the range of the parameter $i$ considerably.

\begin{lemma}\label{classes}
For $i \in \mathbb{Z}$, let $\cF_i$ be the function field as given in Equation \eqref{curveC}. Then $\cF_i$ is isomorphic to $\cF_{i \mod m}$, where $i \mod m$ denotes the remainder of $i$ modulo $m$. 
\end{lemma}
\begin{proof}
For $i \in \mathbb{Z}$, write $i=am+r$ where $r \equiv i \pmod m$ and $a \in \Z$. 
Then 
$$
\varphi_{i,r} : \begin{cases} \cF_i \longrightarrow \cF_r, \\ (x,y) \mapsto (x,y/x^a)=:(x,\tilde{y}). \end{cases}
$$ 
is a field homomorphism from $\cF_i$ to $\cF_r$ over $\mathbb{F}_{q^2}$. Indeed, one directly verifies that
$$\tilde{y}^m=\frac{y^m}{x^{am}}=\frac{x^i(x^2+1)}{x^{am}}=x^r(x^2+1).$$
Moreover, the map $\varphi_{i,r}$ is an isomorphism, since its inverse is the map $(x,\tilde{y}) \mapsto (x,\tilde{y} x^a)$.
\end{proof}

Lemma \ref{classes} allows us, without loss of generality, to assume that $0 \le i \le m-1$ when studying isomorphism classes. As a matter of fact, we may assume that $1 \le i \le m-1$, since for $i=0$ the condition $\gcd(i,m)=1$ is not satisfied. It turns out that a further restriction is possible as demonstrated in the following lemma. 

\begin{lemma}\label{isom2}
Let $1 \leq i \leq m-3$. Then $\cF_i$ is isomorphic to $\cF_{m-2-i}$.
\end{lemma}

\begin{proof}
It is enough to consider the field homomorphism
$$\varphi_{m-2-i,i} : \begin{cases} \cF_{m-2-i} \longrightarrow \cF_{i}, \\ (x,y) \mapsto (1/x,y/x)=:(\tilde{x},\tilde{y}). \end{cases}$$
This map is well-defined as
$$\tilde{y}^m=\frac{y^m}{x^m}=\frac{x^{m-2-i}(x^2+1)}{x^m}=\frac{1}{x^i}+\frac{1}{x^{i+2}}=\frac{1}{x^i}\bigg( 1+\frac{1}{x^2}\bigg)=\tilde{x}^i(\tilde{x}^2+1).$$
Moreover, the map is actually a field isomorphism, as its inverse is simply given by $\varphi_{i,m-2-i}$.
\end{proof}

Since for $i=m-2$, the condition $\gcd(i+2,m)=1$ is not fulfilled, Lemmas \ref{classes} and \ref{isom2} together imply that we are left to study the isomorphism classes of the function fields $\cF_i$ for $i=1, \dots, \floor{\frac{m-2}{2}}$ and $i=m-1$, always assuming that $\gcd(i,m)=\gcd(i+2,m)=1$. To study these function fields further, we will first obtain information on the Weierstrass semigroups of various places of these function fields. To completely determine all isomorphism classes, we will also need information on their automorphism groups, but this is the topic of a later section. 

\subsection{The semigroups of the places $P_0, P_{\pm \alpha}$ and $P_\infty$}

Our first observation is that for a given $i \in \Z$ the places $P_{\alpha}$ and $P_{-\alpha}$ have the same semigroup, though it may depend on $i$. Indeed, this follows directly from the description of the gapsequence of these places from \cite[Propositions 4.1, 4.2]{CMQ2023}, but it is also not hard to see that $\cF_i$ has an automorphism interchanging $P_{\alpha}$ and $P_{-\alpha}$, see Section \ref{sec:aut} for more details. For this reason, we will sometimes with slight abuse of notation talk about the semigroup of $P_{\pm \alpha}$, denoted by $H(P_{\pm \alpha})$. Also the gapsequences of the places $P_0$ and $P_\infty$ can be described using  \cite[Propositions 4.1, 4.2]{CMQ2023}, but it will be more convenient for us to describe a set of generators of these semigroups instead. We do this in the following theorem. For $r \in \mathbb{R}$, the expression $\lceil r \rceil$, resp. $\lfloor r \rfloor$, denotes the ceiling, resp. floor, of $r$.

\begin{theorem}\label{weierstrass_generators_P0_Pinf_thm}
The Weierstrass semigroups at the places $P_0$ and $P_\infty$ of the function field $\cF_i$ can be generated as follows:
$$H(P_0)=\left\langle m,\left\lceil \frac{\ell(i+2)}{m}\right\rceil m-\ell i \mid 1 \le \ell \le m-1 \right\rangle$$
and 
$$H(P_\infty)=\left\langle m,-\left\lfloor \frac{\ell i}{m}\right\rfloor m+\ell (i+2) \mid 1 \le \ell \le m-1 \right\rangle.$$
\end{theorem}
\begin{proof}
We prove the statement for $H(P_0)$. First of all note that for $k,\ell \in \mathbb{Z}$ and $0 \le \ell \le m-1$, the functions $x^k y^\ell$ form a linearly independent set. From Equation \eqref{divisors}, we have
$$(x^ky^\ell)=(km+\ell i)P_0+\ell(P_\alpha+P_{-\alpha})-(km+\ell(i+2))P_\infty.$$
In particular, such a function $x^ky^\ell$ has no poles, except possibly at $P_0$, if and only if $km+\ell(i+2) \le 0$. 
For $n \in \mathbb{Z}$, define the linear space 
$$L_n= \langle x^ky^\ell \mid 0 \le \ell \le m-1, km+\ell(i+2) \le 0, km+\ell i \ge -n\rangle \subseteq L(nP_0),$$
where $L(nP_0)$ is the Riemann-Roch space associated to $nP_0$. If $L_{n+1} \neq L_n$, then there exists a function of the form $f=x^ky^\ell$ such that $v_{P_0}(f)=-(n+1)$. In particular, $L_{n+1} \neq L_n$ implies that $L((n+1)P_0) \neq L(nP_0)$. 

We claim that $L_n = L(nP_0)$ for all $n$. If $n<0$ then $L_n=\{0\}=L(nP_0)$, and $L_0=\mathbb{F}_{q^2}=L(0P_0)$. From the description of $L_n$ one directly can deduce that for $n>2m$, the vector space $L_n$ is generated by the functions $x^ky^\ell$, with $0 \le \ell \le m-1$ and $(k,\ell)$ lying on the boundary or inside the polygon with vertices $(0,0)$, $(-(i+2),m)$, $(-n/m,0)$ and $(-i-n/m,m)$. This polygon has area $n-m$. Moreover, if $n$ is a multiple of $m$, the boundary of the polygon contains $2n/m$ many points with integer coordinates, of which $n/m-1$ satisfy $\ell=m$. Hence for $n>2m$ a multiple of $m$, we deduce from the above and Pick's theorem that 
$$\dim \, L_n = n-m-(2n/m)/2+1+(n/m+1)=n-m+2.$$
From the Riemann-Roch theorem for any $n>2m$ we have $\dim \, L(nP_0)=n+1-g=n-m+2$. We conclude that if $n>2m$ is a multiple of $m$, then $L_n=L(nP_0).$

To finish the proof of the claim that $L_n=L(nP_0)$ for all $n$, it is now enough to show that $\dim \, L_{n+1} -\dim \, L_n \in \{0,1\}$ for all $n$. More precisely, this is enough, since $\dim \, L((n+1)P_0) -\dim\, L(nP_0) \in \{0,1\}$, $L_{n+1} \neq L_n$ implies that $L((n+1)P_0) \neq L(nP_0)$, and we already have shown that $L_0=L(0P_0)$ and $L_n=L(nP_0)$ for infinitely many $n>0$. Now, note that since $\gcd(i,m)=1$, if $k_1m+\ell_1 i=k_2m+\ell_2 i$ for $k_1,k_2, 0 \leq \ell_1,\ell_2 \leq m-1$ in $ \mathbb{Z}$, then $\ell_1=\ell_2$ and $k_1=k_2$.
This shows that for any $n \in \mathbb{Z}$ we have $$\dim \, L_{n+1} -\dim \, L_n \in \{0,1\}.$$

Now that we know that $L_n=L(nP_0)$ for all integers $n$, we can conclude that the ring $\bigcup_{n \ge 0} L(nP_0)$ is generated as an $\mathbb{F}_{q^2}$-vector space by the set of functions 
$$\{x^ky^\ell \mid 0 \le \ell \le m-1, km+\ell(i+2) \le 0\}.$$
This implies that this ring is generated as an $\mathbb{F}_{q^2}$-algebra by the set of functions
$$\{x^{-1}\} \cup \{x^{-\left\lceil \frac{\ell(i+2)}{m} \right\rceil}y^\ell \mid 0 \le \ell \le m-1\}.$$
Hence the semigroup $H(P_0)$ is generated by the corresponding pole orders: 
$$m \quad \text{and} \quad \left\lceil \frac{\ell(i+2)}{m}\right\rceil m-\ell i, \quad \text{where} \quad 0 \le \ell \le m-1.$$
Since for $\ell=0$, this pole order is simply $0$, it can be removed when describing a generating set. \end{proof}

\begin{remark}\label{rem:Apery}
The proof of Theorem \ref{weierstrass_generators_P0_Pinf_thm} actually implies that the elements of $H(P_0)$ correspond exactly to the pole orders of functions of the form $x^{-j}\cdot x^{-\left\lceil \frac{\ell(i+2)}{m} \right\rceil}y^\ell \in \cF_i$, where $j \ge 0$ and $0 \le \ell \le m-1$. 
Since $x^{-1}$ has pole order $m$ at $P_0$, this implies that each of the listed generators of $H(P_0)$ in Theorem \ref{weierstrass_generators_P0_Pinf_thm}, except for $m$ itself, is the smallest possible element of $H(P_0)$ in its congruence class modulo $m$. Hence, the displayed generators form an Ap\'ery set of $H(P_0)$ with respect to $m$. Exactly the same is true for the listed generators for $H(P_\infty)$: they form an Ap\'ery set of $H(P_\infty)$ with respect to $m.$
\end{remark}

\begin{remark}
A set of generators of the semigroup at $P_\infty$ was also computed in \cite[Theorem 3.2 ]{M2023}. The resulting set of generators is 
$$H(P_\infty)=\left\langle m,i+2, m\ell-(i+2) \left\lfloor \frac{ (\ell-2) m}{i} \right\rfloor \mid 3 \le \ell \le i+1 \right\rangle.$$ 
These generators do not form an Ap\'ery set with respect to $m$, but have the advantage of being a smaller set of generators.
In the proof of Theorem \ref{weierstrass_generators_P0_Pinf_thm}, we computed bases of several Riemann-Roch spaces. We would also like to point out that alternative bases for such Riemann-Roch spaces in Kummer extensions can be obtained using \cite[Theorem 2]{HY2018} and \cite[Theorem 3.1]{GN2022}. 
\end{remark}

The following corollary indicates several gaps of the semigroup $H(P_\infty)$ and will be useful later on.

\begin{corollary}\label{divisors_of_m_are_gaps_cor}
    Any divisor $d > 0$ of $m$, different from $m$, is a gap at $P_\infty$.
\end{corollary}

\begin{proof}
    Let $d<m$ be a positive divisor of $m$ and assume that $d \in H(P_\infty)$. As noted in Remark \ref{rem:Apery}, each generator given in Theorem \ref{weierstrass_generators_P0_Pinf_thm}, except for $m$, is the smallest element in the semigroup in its congruence class modulo $m$. In particular, since $d < m$, there is an integer $\ell\in \{1, \dots, m-1\}$ such that 
    $$
        d = - \left\lfloor \frac{\ell i}{m} \right\rfloor m + \ell(i+2). 
    $$

    Now, $\gcd(i+2,d) \leq \gcd(i+2,m) = 1$, so $d$ divides $\ell$. Hence we may write $\ell' = \frac{\ell}{d}$, for some positive integer $\ell'$. Dividing the above equation by $d$ and rearranging yields
    $$
        \ell' (i+2) = 1 + \left \lfloor \frac{\ell i}{m} \right \rfloor \frac{m}{d} \leq  1 + \frac{\ell i}{d} = 1 + \ell' i.
    $$
    This implies $2\ell' \leq 1$, which is a contradiction, and we conclude that $d$ is a gap for $P_\infty$.
\end{proof}

\subsection{Distinguishing the Weierstrass semigroup at $P_{\pm\alpha}$ from the Weierstrass semigroups at $P_\infty$ and $P_0$}

In this subsection, we will show that in most cases the semigroup at $P_{\pm\alpha}$ is distinct from both the semigroup at $P_\infty$ and the semigroup at $P_0$. The first step towards proving this is the following lemma describing the holomorphic differentials of $\cF_i$.

\begin{lemma}\label{lem:homomorphic_diff}
Let $i \in \Z$ satisfy $\gcd(i,m)=\gcd(i+2,m)=1$.
Further, for $k,\ell \in \mathbb{Z}$, denote by $\omega_{k,\ell}$ the differential on $\cF_i$ given as follows: 
$$\omega_{k,\ell}=x^{k-1}y^{\ell-m}dx.$$
Then the space of holomorphic differentials on $\cF_i$ has basis
$$\{\omega_{k,l} \mid \ell > 0, km+\ell i > mi \text{ and }  km+\ell (i+2) < m(i+2)\}.$$
\end{lemma}
\begin{proof}
For $k,\ell \in \mathbb{Z}$, consider the differential 
$$\omega_{k,\ell}=x^{k-1}y^{\ell-m}dx.$$ 
From Equation \eqref{divisors}, we have
$$(\omega_{k,\ell})=(km+\ell i-mi-1)P_0+(\ell-1)(P_\alpha+P_{-\alpha})-(km+\ell(i+2)-m(i+2)+1)P_\infty.$$
This means that the differential $\omega_{k,\ell}$ has no poles if and only if $\ell > 0$, $km+\ell i > mi$ and  $km+\ell (i+2) < m(i+2).$ Note that the number of pairs $(k,\ell)$ satisfying these conditions is precisely the number of integral points, i.e., points in $\mathbb{Z}^2$, lying inside the triangle $\Delta$ with vertices $(i,0)$, $(i+2,0)$ and $(0,m)$. Pick's theorem implies that there are precisely $m-1$ such points $(k,\ell) \in \mathbb{Z}^2$ in the interior of $\Delta$. 
Here we used the condition $\gcd(i,m)=\gcd(i+2,m)=1$ to see that there are exactly four integral points on the boundary of $\Delta$, namely the points $(i,0)$, $(i+1,0)$ $(i+2,0)$ and $(0,m)$. 

We can conclude that we have found $m-1$ holomorphic differentials. Since any $(k,\ell)$ in the interior of $\Delta$ satisfies $1 \le \ell \le m-1$ and the defining equation of $\cF_i$ in Equation \eqref{curveC} has $y$-degree $m$, these $m-1$ differentials are linearly independent. Moreover, since the genus of $\cF_i$ is $m-1$, we see that we have found a basis of the space of holomorphic differentials.
\end{proof}

\begin{remark}\label{gaps_at_P0_and_Pinf_lemma}
From Proposition \ref{propgaps}, the value $v_P(\omega)+1$ is a gap for $P$ for any place $P$ of $\cF_i$ and any holomorphic differential $\omega$ on $\cF_i$. Using Lemma \ref{lem:homomorphic_diff}, we immediately obtain that the gaps at the places $P_0$ and $P_\infty$ of the function field $\cF_i$ are 
\begin{align*}
&G(P_0)=\{ km+\ell i-mi \mid  \ell > 0, km+\ell i > mi, km+\ell (i+2) < m(i+2)\}\\
&\text{and}\\
&G(P_\infty)=\{-km-\ell (i+2)+m(i+2) \mid  \ell > 0, km+\ell i > mi, km+\ell (i+2) < m(i+2)\}.
\end{align*}
\end{remark}

Similarly, we get the following result regarding $G(P_{\pm\alpha})$.

\begin{lemma}\label{gapsequence_Palph_lemma}
Let $i \in \Z$ satisfy $\gcd(i,m)=\gcd(i+2,m)=1$. The gapsequence of the place $P_{\pm \alpha}$ of $\cF_i$ contains the numbers $1, \dots, \floor{\frac{m+1}{2}}$. 
\end{lemma}
\begin{proof}
    Using the holomorphic differentials given in Lemma \ref{lem:homomorphic_diff} combined with the fact that $v_{P_{\pm\alpha}}(\omega_{k,\ell}) + 1 = \ell$, we see that $\ell$ is a gap at $P_{\pm\alpha}$ for any pair $(k,\ell) \in \mathbb{Z}^2$ such that $\ell> 0$, $km+\ell i > mi$, and $km + \ell(i+2) < m(i+2)$.

    In particular, a fixed $\ell$ satisfying $0< \ell < m$ is a gap if there exists an integer $k$ in the interval
    $$
        \left\lbrack \frac{i(m-\ell)}{m}, \frac{(i+2)(m-\ell)}{m} \right\rbrack.
    $$
   Note that we may include the endpoints in the interval, since neither of them are integers, as follows from the assumptions $0 < \ell <m$ and $\gcd(i,m)=\gcd(i+2,m)=1$. 
   
   Now, assume $0 < \ell \leq  \lfloor\frac{m+1}{2}\rfloor$ and write $i(m-\ell) = s m + r$ for $s, r \in \mathbb{Z}$, with $0 < r < m$. We claim that $s+1$ is in the above interval. Indeed, we see that 
   $$
        \frac{i(m-l)}{m} = s + \frac{r}{m} < s + 1,
   $$
   and 
   $$
        \frac{(i+2)(m-\ell)}{m} = s + \frac{r + 2(m-\ell)}{m} \geq s + \frac{1 + 2\left(m-\frac{m+1}{2}\right)}{m} = s + 1.
   $$
\end{proof}

\begin{theorem}\label{distinguishing_Palph_Pinf_cor}
    For $1 \le i < \frac{m-2}{2}$, the Weierstrass semigroups at the places $P_{\pm\alpha}$ and $P_\infty$ of $\cF_i$ are distinct.
\end{theorem}
\begin{proof}
    Using the formula for the divisor of $y$ as given in Equation \eqref{divisors}, it follows directly that $i+2 \in H(P_\infty)$. On the other hand, Lemma \ref{gapsequence_Palph_lemma} implies that $i+2 \notin H(P_{\pm\alpha})$ since $i < \frac{m-2}{2}$ implies $i+2 \leq \frac{m + 1}{2}$.
\end{proof}

The rest of this subsection is devoted to showing that the semigroups of $P_0$ and $P_{\pm\alpha}$ are distinct for $i \notin \{1, \frac{m-2}{2} \}$.

\begin{lemma}\label{P0_polenumber_lessthan_lemma}
Let $m\ge 19$ and suppose that $i$ satisfies $\gcd(i,m)=\gcd(i+2,m)=1$ and $2<i<\lfloor\frac{m-3}{2}\rfloor$. Then $H(P_0)$ contains a positive integer less than or equal to $\frac{m+1}{2}$.
\end{lemma}

\begin{proof}
    From Theorem \ref{weierstrass_generators_P0_Pinf_thm} we see that $\lceil \frac{\ell(i+2)}{m}\rceil m - \ell i$ is a pole number for $P_0$, for any $1 \leq \ell \leq m-1$. By choosing $\ell = \lfloor \frac{m}{i+2} \rfloor$ this becomes $m- \lfloor \frac{m}{i+2} \rfloor i $. 

    Writing $m = \ell (i+2) + r$ for some $r \in \mathbb{Z}$, with $0 < r < (i+2)$, we get the bound
    $$
        m- \left\lfloor \frac{m}{i+2} \right\rfloor i = m- i \left(\frac{m}{i+2} - \frac{r}{i+2} \right) 
            \leq  m - \frac{im}{i+2} + \frac{i(i+1)}{i+2} 
             = \frac{2m + i(i+1)}{i+2}.
    $$

    In most cases, this bound is strong enough to show what we want. Indeed, 
    $$
        \frac{2m + i(i+1)}{i+2} \leq \frac{m+1}{2} \Leftrightarrow i^2 - \frac{m-1}{2} i + m-1 \leq 0,
    $$
    and for $m\geq 19$ we see that this is true when $i$ satisfies $2< i < \lfloor \frac{m-3}{2} \rfloor - 1$.

    The last case, $i = \lfloor \frac{m-3}{2} \rfloor - 1$, can be handled in a very similar way. Choosing again $\ell = \lfloor \frac{m}{i+2} \rfloor =  2$ we get a pole number 
    $$
        m- \left\lfloor \frac{m}{i+2} \right\rfloor i = m - 2 \left( \left\lfloor \frac{m-3}{2} \right\rfloor - 1 \right) \leq m - 2 \left( \frac{m-3}{2} -2 \right) = 7, 
    $$
    Since $7 \leq \frac{m+1}{2}$ for $m\geq 19$ (even for $m \geq 13$), this finishes the proof.
\end{proof}

For even $m$, the above is enough to distinguish the Weierstrass groups of $P_{\pm\alpha}$ and $P_0$ for any $i < \frac{m-2}{2}$, since $\lfloor \frac{m-3}{2} \rfloor - 1 = \frac{m-2}{2} - 2$ and $i = \frac{m-2}{2} - 1$ is in contradiction with $(i+2,m)=1$. For odd $m$, we still need to treat the cases $i=2$ and $i = \frac{m-3}{2}$. 

\begin{lemma}\label{Palph_P0_specialcase_lemma}
    For $i \in \{2, \frac{m-3}{2}\}$, $m$ odd, and $m > 9$, it holds that $1, \dots, \frac{m+3}{2}$ are gaps for $P_{\pm\alpha}$ while $P_0$ has a pole number less than or equal to $\frac{m+3}{2}$.
\end{lemma}
\begin{proof}
    We already know from Lemma \ref{gapsequence_Palph_lemma} that $1, \dots, \frac{m+1}{2}$ are gaps for $P_{\pm \alpha}$. To show that $\frac{m+3}{2}$ is also a gap, we consider again the interval from the proof of Lemma \ref{gapsequence_Palph_lemma}. With $\ell = \frac{m+3}{2}$ our task becomes to find an integer inside the interval 
    $$
        \left\lbrack \frac{i(m-3)}{2m}, \frac{(i+2)(m-3)}{2m} \right\rbrack.
    $$
    
    For $i=2$, this simplifies to $\lbrack \frac{m-3}{m}, \frac{2(m-3)}{m}\rbrack$, which contains $1$ for $m>3$. For $i=\frac{m-3}{2}$ the interval becomes 
    $$
        \left\lbrack \frac{(m-3)^2}{4m}, \frac{(m-3)(m+1)}{4m}\right\rbrack,
    $$
    which contains $\frac{m-3}{4}$ when $m \equiv 3 \pmod 4$ and $\frac{m-5}{4}$ when $m \equiv 1 \pmod 4$ and $m > 9$. 

    To show that $P_0$ has a pole number less than or equal to $\frac{m+3}{2}$ we use the same method as in Lemma \ref{P0_polenumber_lessthan_lemma}. For $i=2$ we choose $\ell = \lfloor \frac{m}{4} \rfloor$ and write $m = 4s + r$ for $s,r \in \mathbb{Z}$, with $0 < r <4$. Then, we have a pole number
    $$
        m - \left\lfloor \frac{m}{4} \right\rfloor i = m - 2s \leq \frac{m+3}{2},
    $$
    where the inequality follows from $s = \frac{m-r}{4} \geq \frac{m-3}{4}$. Similarly, for $i = \frac{m-3}{2}$, we choose $\ell = \lfloor \frac{m}{i+2} \rfloor = \lfloor \frac{2m}{m+1} \rfloor = 1$, and get a pole number 
    $$
        m - i = \frac{m+3}{2},
    $$
    as wished.
\end{proof}

By combining the above results and checking the cases $m < 19$ with a computer we get:

\begin{theorem} \label{differentP0Pal}
    For $i$ satisfying $1<i<\frac{m-2}{2}$ and $\gcd(i,m) = \gcd(i+2,m) = 1$, the places $P_0$ and $P_{\pm \alpha}$ of $\cF_i$ have distinct Weierstrass semigroups.
\end{theorem}   

\begin{remark}\label{rem:special_i}
Theorems \ref{distinguishing_Palph_Pinf_cor} and \ref{differentP0Pal} together cover the case $1<i<(m-2)/2$. To complete the picture, let us describe what happens for $i=1$ and $i=(m-2)/2$. 

If $i=1$, we have $\gcd(3,m)=1$ and using this and  \cite[Propositions 4.1, 4.2]{CMQ2023}, one quickly sees that $G(P_0)=G(P_\a)=G(P_{-\a})=\{1,\dots,\lfloor 2m/3 \rfloor\} \cup \{m+1,\dots,m+\lfloor m/3 \rfloor\}.$ Hence we see that $P_0$, $P_{\a}$ and $P_{-\a}$ all have the same Weierstrass semigroup. From Theorem \ref{distinguishing_Palph_Pinf_cor}, we already know that $P_{\infty}$ has a different semigroup. In fact $H(P_\infty)=\langle 3,m\rangle$ as can be seen from for example Theorem \ref{weierstrass_generators_P0_Pinf_thm}.

If $i=(m-2)/2$, all four places $P_\infty$, $P_0$, $P_{\a}$ and $P_{-\a}$ turn out to have the same Weierstrass semigroup. Indeed, as we will see in Subsection \ref{subsec:m-2/2}, the function field $\cF_{(m-2)/2}$ has an automorphism acting as a $4$-cycle on these four places. Note that if $i=(m-2)/2$, then $m$ needs to be even and $i$ needs to be odd, since $\gcd(i,m)=1$. Then in fact $m$ needs to be a multiple of four. Using this, it is not hard to use Theorem \ref{weierstrass_generators_P0_Pinf_thm} to show that the semigroup of places $P_\infty$, $P_0$, $P_{\a}$ and $P_{-\a}$ is $\langle \frac{m}{2}+1,\frac{m}{2}+3,m-3,m-1,m\rangle$. Indeed, considering the generators in the theorem for $H(P_\infty)$ for $\ell=2\ell'+1 < m/2$, one finds the generators $\frac{m}{2}+1,\frac{m}{2}+3,m-3,m-1$. These together with $m$ already generate a semigroup with $m-1$ gaps, so no more generators are needed.
\end{remark}

\section{The automorphism group of $\cF_i$}\label{sec:aut}

We denote by $\aut(\cF_i)$ the $\overline{\mathbb{F}}_{q^2}$-automorphism group of $\overline{\mathbb{F}}_{q^2}\cF_i$. For convenience we will simply call $\aut(\cF_i)$ the automorphism group of $\cF_i$.
As a first observation note that the  automorphism group of $\cF_i$ contains a cyclic subgroup, which we will denote by $G_i$, with $q+1$ elements. In fact  for $i$ even one has
$$G_i:=\{\sigma: (x, y)\mapsto (ax, by)\mid a, b \in \fqs , a^2=b^{\frac{q+1}{2}}=1 \}\subseteq \aut(\cF_i),$$  
and for $i$ odd 
$$G_i:=\{\sigma: (x, y)\mapsto (ax, by)\mid a, b \in \fqs, a^2=1, b^{\frac{q+1}{2}}=a \}\subseteq \aut(\cF_i).$$ 
Our aim is to show the automorphism group of $\cC_i$ quite often coincides with $G_i$ itself.

The fixed field of $G_i$ is given by $\mathbb{F}_{q^2}(x^2)$ and it is easy to see that the only ramified places of $\cF_i$ in the extension $\cF_i/\mathbb{F}_{q^2}(x^2)$ are the places $P_0$, $P_\infty$ and $P_{\pm \a}.$ In other words, the only short orbits of $G_i$ are $\{P_\infty\}$, $\{P_0\}$ and $\{P_\a,P_{-\a}\}.$  
The subgroup
$$H_i:=\{\sigma: (x, y)\mapsto (x, by)\mid b^{\frac{q+1}{2}}=1,  b \in \fqs \}\subseteq G_i$$ has index two in $G_i$ and contains the elements of $G_i$ fixing both $P_\al$ and $P_{-\a}$.

Our aim is to show the automorphism group of $\cF_i$ is not particularly larger than $G_i$, and quite often coincides with $G_i$ itself.
Note that since $\cF_i$ is isomorphic to $\cF_{m-2-i}$, we can assume w.l.o.g. that $1 \le i \leq \lfloor (m-2)/2 \rfloor$ or $i=m-1$. We will deal with the cases $i=(m-2)/2$ and $i=m-1$
later and first assume that $1 \le i \leq (m-3)/2$.

\begin{remark}
It turns out that there are two very special, exceptional cases, that one needs to deal with separately. These cases are $(i,q)=(1,3)$ and $(i,q)=(1,7)$.

If $q=3$ and $i=1$, the function field $\cF_1$ is elliptic. In particular, this implies that its automorphism group is infinite by \cite[Theorem 11.94 (i)]{HKT2008}. 

If $q=7$ and $i=1$, the function field $\cF_1$ has genus three and an automorphism group of order $96$. This group is the semidirect product of a solvable group of order $48$ and a cyclic group of order $2$.
These results can be obtained using a computer, for example the computer-algebra package Magma.  
\end{remark}

\subsection{The case $1 \le i < (m-2)/2$}

The key to determining the automorphism group of $\cF_i$ is the following theorem.

 \begin{theorem} \label{actiononOmega}
Let $i=1,\ldots, (m-3)/2$ with $\gcd(i,m)=\gcd(i+2,m)=1$ and assume $(i,q) \not\in \{(1,3),(1,7)\}$. Then the $\aut(\cF_i)$-orbit containing $P_\infty$ is contained in $\Omega:=\{P_\infty, P_\alpha, P_{-\alpha},P_0\}$. 
 \end{theorem}

 \begin{proof}
If $q\le 31$, one checks the lemma using a computer. We will assume from now on that $q>31$. Denote with $g:=m-1=(q-1)/2$ the genus of $\cF_i$. Assume by contradiction $O_\infty$, the $\aut(\cF_i)$-orbit containing $P_\infty$, is not contained in $\Omega$. From the fact that $G_i$ has only short orbits contained in $\Omega$, while it acts on $O_\infty$, we deduce that $O_\infty$ contains $P_\infty$ and at least one long orbit of $G_i$. Hence
\begin{equation} \label{eq_Oinf}
    |O_\infty| \geq 1+|G_i|=q+2
\end{equation}
     
     Recalling that $G_i$ fixes $P_\infty$ we have from the Orbit-Stabilizer theorem that
     $$|\aut(\cF_i)| = |O_\infty||\aut(\cF_i)_{P_\infty}| \geq (q+2)(q+1)=(2g+3)(2g+2)>84(g-1),$$
     as $q>31$. From \cite[Theorem 11.56]{HKT2008} $\aut(\cF_i)$ has order divisible by the characteristic $p$ and one of the following cases occurs:
     \begin{enumerate}
         \item $\aut(\cF_i)$ has exactly one short orbit.
         \item $\aut(\cF_i)$ has exactly 3 short orbits of which two have cardinality $|\aut(\cF_i)|/2$.
         \item $\aut(\cF_i)$ has exactly 2 short orbits, of which at least one is non-tame (i.e. the stabilizer of a place in the orbit has order divisible by $p$).
     \end{enumerate}
     Note that since all the places in $\Omega$ are fixed by $H_i$, they need to be contained in some short orbits of $\aut(\cF_i)$. From Corollary \ref{distinguishing_Palph_Pinf_cor}, the orbit $O_\infty$ cannot contain $P_\alpha$ and hence $\aut(\cF_i)$ must have at least two short orbits, namely $O_\infty$ and $O_\alpha$, the orbit containing both $P_\alpha$ and $P_{-\alpha}$. This implies that Case (1) above cannot occur. Also Case (2) cannot occur, as both $O_\infty$ and $O_\alpha$ have a stabilizer of order at least $m$, and hence larger than $2$. The orbit stabilizer theorem implies that both $|O_\infty|$ and $|O_\alpha|$ are at most $|\aut(\cF_i)|/m<|\aut(\cF_i)|/2$. 

     This proves that Case (3) holds necessarily and hence $O_\infty$ and $O_\alpha$ are exactly the short orbits of $\aut(\cF_i)$ and at least one of these two orbits has a stabilizer of order divisible by $p$. We denote with $O$ such an orbit, without specifying whether $O=O_\infty$ or $O=O_\alpha$.

Let $P \in O$ and denote with $S_p$ the Sylow $p$-subgroup of $\aut(\cF_i)_P$ (depending on which orbit $O$ is, one could choose $P=P_\infty$ or $P=P_\alpha)$. From \cite[Theorem 11.49]{HKT2008} we can write $\aut(\cF_i)_P=S_p \rtimes C$, where $C$ is a cyclic group of order prime to $p$ containing $H_i$ (recall that $H_i$ fixes all the places in $\Omega$). 

Denote with $\widetilde{\cF}_i$ the fixed field of $\cF_i$ with respect to $S_p$ and by $g(\widetilde{\cF}_i)$ its genus. We distinguish three cases: either $g(\widetilde{\cF}_i) \geq 2$, $g(\widetilde{\cF}_i)=1$ or $g(\widetilde{\cF}_i)=0$. 

\bigskip
\noindent
\textbf{Case 1 $(g(\widetilde{\cF}_i)\geq 2)$}: The quotient group $\aut(\cF_i)_P/S_p \cong C$ is a cyclic automorphism group of $\widetilde{\cF}_i$ of order at least $m=g+1$. From \cite[Theorem 11.79]{HKT2008} $m \leq |C| \leq 4g(\widetilde{\cF}_i)+4$. However the Hurwitz genus formula implies that $g-1 \geq |S_p|(g(\widetilde{\cF}_i)-1)+(|S_p|-1)=|S_p|g(\widetilde{\cF}_i)-1$ and so
    $$g(\widetilde{\cF}_i) \leq \frac{g}{|S_p|}.$$
    Combining all the above, we need to have that
    $$g+1=m \leq |C| \leq 4g(\widetilde{\cF}_i)+4 \leq  \frac{4g}{|S_p|}+4.$$
    Hence
    $$p \leq |S_p| \leq  \bigg\lfloor\frac{4g}{g-3}\bigg\rfloor=4.$$
In the final equality, we used that $g=m-1>15$, which follows since we assumed $q>31$.
    So this case is only possible if $|S_p|=p=3$, which in turn implies that $|C|=m$, since we now know that $|C| \le 4g/3+4$. Hence $C=H_i$. We want to prove that also this case is actually not possible. To do so, first note that since $G_i$ fixes both $P_\infty$ and $P_0$ (and has order larger than $m$) this also implies that $O=O_\alpha$ and $P_0 \not\in O$. Recall that $C=H_i$ fixes $P=P_\alpha$ and also $P_{-\alpha}$ and normalizes $S_3$. Hence $H_i$ needs to act on the $S_3$-orbit containing $P_{-\alpha}$. 
    Since $\cF_i$ has $3$-rank zero, we know from \cite[Lemma 11.129]{HKT2008} that the generator of $S_3$ fixes $P=P_\alpha$ and acts on $O \setminus \{P\}$ with orbits of length three. This means that necessarily the $S_3$-orbit containing $P_{-\alpha}$ has length three. This is not possible because $H_i$ fixes the places in $\Omega$ and acts with orbits of length $m$ elsewhere ($m>3$).  This gives the desired contradiction.
    
\bigskip
\noindent
\textbf{Case 2 $(g(\widetilde{\cF}_i)=1)$}: As before the quotient group $\aut(\cF_i)_P/S_p \cong C$ is a cyclic automorphism group of $\widetilde{\cF}_i$ of order at least $m=g+1$. Clearly such a group fixes at least one place of $\widetilde{\cF}_i$, namely the one below $P$. From \cite[Remark 11.95]{HKT2008} it needs to be true that $m \leq |C| \leq 12$, which is not possible for $q>31$.
    
\bigskip
\noindent
\textbf{Case 3 $(g(\widetilde{\cF}_i)=0)$}: The quotient group $\bar C:=\aut(\cF_i)_P/S_p \cong C$ is a cyclic automorphism group of $\widetilde{\cF}_i$ of order at least $m=g+1$ fixing at least one place. From \cite[Theorem 11.91]{HKT2008}, $\bar C$ fixes exactly two places in $\widetilde{\cF}_i$. In other words: $H_i$ fixes exactly two $S_p$-orbits, one of which is the orbit $\{P\}$. On the other hand, the only short orbits of $H_i$ are the places in $\Omega$ implying that any $S_p$-orbit fixed by $H_i$ is contained in $\Omega$. Hence the second $S_p$-orbit that is fixed by $H_i$ needs to be equal to $\Omega \setminus \{P\}$. Theorem \ref{distinguishing_Palph_Pinf_cor} implies that $P=P_\infty$ and therefore that $\Omega \setminus \{P\}=\{P_0,P_\a,P_{-\a}\}.$ In particular, $p=3$ and $|S_3|=3$, since $\cF_i$ has $3$-rank zero. 
Moreover we can conclude that $H(P_\alpha)=H(P_0)$, which implies that $i=1$ from Theorem \ref{differentP0Pal}. 

Note that $\cF_1$ is the function field of a plane curve given by a separated polynomial. Hence by \cite[Theorem 12.11]{HKT2008}, $\aut(\cF_1)$ fixes $P_\infty$ unless $\cF_1$ is isomorphic to an Hermitian function field, or to a function field of type $\overline{\mathbb{F}}_{q^2}(u,v)$ with $v^{\bar m}=u^{3^t}+u$ for some $\bar m$ and $t$ satisfying $\bar m \mid (3^t+1)$. Since the genus of a Hermitian function field would in this case be divisible by $3$ (while $m-1$ is not), this case cannot occur. If $\cF_1$  is isomorphic to a function field of type $\overline{\mathbb{F}}_{q^2}(u,v)$ as described above, 
then by comparing their genera, it needs to be true that
    $$2g=2(m-1)=q-1=3^n-1=g(\mathbb{F}_{q^2}(u,v))=(3^t-1)(\bar m-1).$$
    Since this implies that $3^t-1$ divides $3^n-1$ we deduce that necessarily $t \mid n$, that is, $n=st$ for some $s \geq 1$. Combining with the fact that $\bar m$ divides $3^t+1$ we get
    $$3^n-1=3^{st}-1=(3^t-1)(\bar m-1) \leq (3^t-1)3^t<3^{2t}-1.$$
    This proves that $s=1$, that is $n=t$, and hence $\bar m=2$. This implies that $\overline{\mathbb{F}}_{q^2}\cF_i$ has a place $P$ such that $2 \in H(P)$. On the other hand, by Remark \ref{rem:special_i}, $2$ is a gap of $P \in \Omega$. The same is true for $P \not\in \Omega$ as we show now. Consider the differential 
$$\omega:=(x-a)\frac{dx}{y^{m-1}},$$
where $a:=x(P)$. From Equation \eqref{divisors}, the divisor of $\omega$ is of the form
$$(\omega)=P+E+(m-4)P_\infty,$$
where $E \geq 0$ and $P \notin \supp(E)$. Since $m > 16$, $\omega$ is holomorphic and hence $v_P(\omega)+1=2$ is a gap at $P$ using Proposition \ref{propgaps}. We can conclude that  $\overline{\mathbb{F}}_{q^2}\cF_1$ cannot be isomorphic to $\overline{\mathbb{F}}_{q^2}(u,v)$.
This shows that $\aut(\cF_1)$ fixes $P_\infty$.  This is in contradiction with Equation \eqref{eq_Oinf}.

We have now shown that in all possible cases Equation \eqref{eq_Oinf} gives rise to a contradiction. This proves that the $\aut(\cF_i)$-orbit containing $P_\infty$ is contained in $\Omega$.
 \end{proof}

Using the results from the proof of Theorem \ref{actiononOmega}, we can complete the case $(i,p)=(1,3)$ quite easily.

\begin{corollary} \label{cor_p3i1}
Let $q=3^n$ with $n>1$. Then $\aut(\cF_1)=C_3 \rtimes G_1$ where $C_3$ is the cyclic group of order $3$ generated by the automorphism $\gamma$ defined by $\gamma(x)=x+\alpha$ and $\gamma(y)=y$. 
In particular $|\aut(\cF_1)|=3(q+1)$.
\end{corollary}
\begin{proof}
Recall from the proof of Theorem \ref{actiononOmega} that
$\aut(\cF_1)$ fixes $P_\infty$. 
Therefore, by \cite[Theorem 11.49]{HKT2008}, we can write $\aut(\cF_1)=\aut(\cF_1)_{P_\infty}=S_3 \rtimes C$, where $S_3$ is the Sylow $3$-subgroup of $\aut(\cF_1)$ and $C$ is a cyclic group of order prime to $3$ containing $G_1$. Note that $S_3$ is not trivial, as it contains the automorphism $\gamma$ defined in the statement of the corollary. It is easy to see that $\gamma$ indeed is an automorphism (using $\alpha^2+1=0$) and that it has order three.
From \cite[Theorem 12.7 (i)]{HKT2008} we then have that $|C|=m(3-1)=2m=q+1$ and since $C$ contains $G_1$ (which has order $q+1$) we deduce that $C=G_1$. 
    Since $2m=q+1 \equiv 1 \pmod{3}$, we have $m  \equiv 2 \pmod{3}$ and hence by \cite[Theorem 12.7 (iii)]{HKT2008} we can deduce that $|S_3|=3$. 
    This shows that $|\aut(\cF_1)|=3(q+1)$ and $\aut(\cF_1)=C_3 \rtimes G_1$, with $C_3$ being the cyclic group generated by $\gamma$.   
\end{proof}

Now we are in a position to complete the determination of $\aut(\cF_i)$ for $i=1,\ldots,(m-3)/2$.

\begin{proposition} \label{autoq31}
Let $i=1,\ldots, (m-3)/2$ with $\gcd(i,m)=\gcd(i+2,m)=1$ and  assume $(i,q) \neq (1,7)$ and $(i,p) \ne (1,3)$. Then $\aut(\cF_i)=G_i$.
 \end{proposition}

 \begin{proof}
For $q \le 31$, the proposition can be verified using a computer. Hence we assume $q>31$ from now on. From Theorem \ref{actiononOmega} we know that the $\aut(\cF_i)$-orbit containing $P_\infty$, say $O_\infty$, is contained in $\Omega=\{P_\infty, P_0, P_\alpha, P_{-\alpha}\}$. 
Recalling that from Corollary \ref{distinguishing_Palph_Pinf_cor}, $P_\alpha$ and $P_{-\alpha}$ cannot be contained in $O_\infty$ , we have only the following possibilities
\begin{enumerate}
 \item $O_\infty=\{P_\infty\}$ or
    \item $O_\infty=\{P_0,P_\infty\}$.
\end{enumerate}

If Case (1) occurs then $\aut(\cF_1)=S_p \rtimes C$, where $S_p$ is the Sylow $p$-subgroup of $\aut(\cF_1)$ and $C$ is a cyclic group of order prime to $p$ containing $G_i$. In reality $C=G_i$. The reason is that $G_i$ is trivially a normal subgroup in $C$ (because $C$ is abelian) and hence $C$ acts on the fixed points of $G_i$. Since $C$ already fixed $P_\infty$, this implies that $C$ fixes also $P_0$. Hence a generator $\beta$ of $C$ will map $x$ to $\lambda x$ (for some constant $\lambda$, because it preserves the divisor of $x$). Recall that $C$ also normalizes $H_i$ and hence acts on $\Omega$ (the set of fixed points of $H_i$). This means that $C$ also acts on $\{P_\alpha, P_{-\alpha}\}$ and hence preserves the divisor of $y$ as well. Since $\beta(x)=\lambda x$ and $\beta(y)=\mu y$ and $\beta(y)^m=\beta(x)^i(\beta(x)^2+1)$, we get
$$\mu^m y^m=\mu^m x^i(x^2+1)=\beta(x)^i(\beta(x)^2+1)=\lambda^i x^i(\lambda^2 x^2+1),$$
and hence $\mu^m=\lambda^{i+2}=\lambda^i$. In particular, $\beta \in G_i$, and since $\beta$ was a generator of $C \supseteq G_i$, we deduce $C=G_i$. Let $|S_p|=p^t$.
If $t=0$ then we have $\aut(\cF_i)=G_i$, and we are done. Otherwise $t \geq 1$. 

If $t \ge 1$, we can reason in a very similar way as in the proof of Theorem \ref{actiononOmega}. Indeed, the reasoning is easier because we now know that $C=G_i$ instead of only knowing that $H_i \subseteq C$. Moreover, we now assume $(i,p) \neq (1,3)$, so there are fewer cases to consider.

We are left with Case (2), namely $O_\infty=\{P_\infty, P_0\}$. Note that from the Orbit-Stabilizer theorem 
$$|\aut(\cF_i)|=2|S|,$$
where $S \subseteq \aut(\cF_i)$ is the stabilizer of $P_\infty$. Note that $S$ also fixes $P_\infty$, since $O_\infty$ only contains two places. Since $\cF_i$ has $p$-rank zero, an element of order $p$ fixes exactly one place, from \cite[Lemma 11.129]{HKT2008}. Hence $|S|$ is prime to $p$ and $S$ is cyclic. Note that $G_i \subseteq S$ as $G_i$ fixes both $P_0$ and $P_\infty$. We claim that $S=G_i$. To show that, we fix a generator $\beta$ of $S$. Then $\beta$ will map $x$ to $\lambda x$ (for some constant $\lambda$, because it preserves the divisor of $x$). Recall that $S$ also normalizes $H_i$ and hence acts on $\Omega$ (the set of fixed points of $H_i$). This means that $S$ also acts on $\{P_\alpha, P_{-\alpha}\}$ and hence preserves the divisor of $y$ as well. 
As before this implies that $\beta \in G_i$, and since $\beta$ was a generator of $S \supseteq G_i$, we deduce that $S=G_i$.

In particular we know that $|\aut(\cF_i)|=2(q+1)$ and since $G_i$ is a subgroup of $\aut(\cF_i)$ of index $2$, $G_i$ is normal in $\aut(\cF_i)$. In particular $\aut(\cF_i)$ acts on $\Omega$, as $\Omega$ is the union of the three short orbits of $G_i$.


Let $\gamma \in \aut(\cF_i) \setminus G_i$. Then $\gamma(P_0)=P_\infty$, $\gamma(P_\infty)=P_0$ (otherwise $\gamma \in G_i$). Recalling that $\gamma$ acts on $\Omega$ we see that $\gamma$ acts also on $\{P_\alpha, P_{-\alpha}\}$. So from Equation \eqref{divisors}
$$(y)=iP_0+P_\alpha+P_{-\alpha}-(i+2)P_\infty,$$
we get
$$(\gamma(y))=iP_\infty+P_\alpha+P_{-\alpha}-(i+2)P_0,$$
that is,
$$\bigg( \frac{\gamma(y)}{y}\bigg)=(2i+2)(P_\infty-P_0).$$

Since we have also $(x)=m(P_0-P_\infty)$, we must have that $\gcd(2i+2,m) \in H(P_\infty)$. Since $2i+2 \leq m-3+2=m-1$, $\gcd(2i+2,m)$ is a proper divisor of $m$. This is not possible from Corollary \ref{divisors_of_m_are_gaps_cor}.
\end{proof}

\subsection{The case $i=(m-2)/2$}\label{subsec:m-2/2}

We now study the case $i = \frac{m-2}{2}$. Obtaining a complete description of the automorphism group would be interesting, but it appears to be more difficult than for the other values of $i$. Fortunately, a partial result will be sufficient for our purposes. Note again that the condition $\gcd(i,m)=1$ implies that $i=(m-2)/2$ only occurs in case $m$ is a multiple of four. 

\begin{theorem}\label{aut_bound_special_case_thm}
The automorphism group of the function field $\cF_{(m-2)/2}$ contains a subgroup that is the semidirect product of $G_i$ and a cyclic group of order four.
\end{theorem}

\begin{proof}
It is enough to show that there exists an automorphism $\sigma_4$ of order four such that neither $\sigma_4$ nor $\sigma_4^2$ occurs in $G_{(m-2)/2}$. We claim that this automorphism can be given by 
$$\sigma_4(x)=\alpha\frac{x-\alpha}{x+\alpha} \quad \text{and} \quad \sigma_4(y)=4\beta\frac{y^{(m-2)/2}}{x^{(m-4)/4}(x+\alpha)},$$
where $\beta^2=\alpha$. Note that $\beta \in \mathbb{F}_{q^2},$ since eight divides $q^2-1$ and $\beta^4=\alpha^2=-1$. 

First of all, we show that $\sigma_4$ indeed is an automorphism:

\begin{eqnarray*}
\sigma_4(y)^m & = & (4\beta)^m \dfrac{(y^m)^{(m-2)/2}}{x^{(m^2-4m)/4}(x+\alpha)^{m}}\\
 & = & 4 \alpha^{m/2} \dfrac{x^{(m^2-4m+4)/4}(x+\alpha)^{(m-2)/2}(x-\alpha)^{(m-2)/2}}{x^{(m^2-4m)/4}(x+\alpha)^m}\\
& = & 4 \alpha^{m/2} \left(\dfrac{x-\alpha}{x+\alpha}\right)^{(m-2)/2} \dfrac{x}{(x+\alpha)^2}\\
& = & \left(\dfrac{\alpha x+ 1}{x+\alpha}\right)^{(m-2)/2} \left(\left(\dfrac{\alpha x+1}{x+\alpha}\right)^2+1\right)\\
& = & \sigma_4(x)^{(m-2)/2}(\sigma_4(x)^2+1).
\end{eqnarray*}

A direct calculation shows that this automorphism acts as a $4$-cycle on the places $P_0$, $P_{-\alpha}$, $P_\infty$ and $P_\alpha$ of $\cF_{(m-2)/2}$. This shows that both $\sigma_4$ and $\sigma_4^2$ do not fix $P_\infty$. Since any automorphism in $G_{(m-2)/2}$ fixes $P_ \infty$, we see that both $\sigma_4$ and $\sigma_4^2$ are not in $G_{(m-2)/2}$, which is what we needed to show.
\end{proof}

We get the following result regarding semigroups as an immediate consequence, also see Remark \ref{rem:special_i}.

\begin{corollary}\label{equal_semigroups_cor}
The semigroups of the places $P_0$, $P_{-\alpha}$, $P_\infty$ and $P_\alpha$ of $\cF_{(m-2)/2}$ are the same. 
\end{corollary}

\begin{proof}
This is clear, since $\sigma_4$ acts transitively on these places.
\end{proof}

\subsection{The case $i=m-1$}

In this subsection we study the case $i=m-1$ and show that this case is very special. In fact we will see that $\cF_{m-1}$ is isomorphic to the famous Roquette function field:
$$\mathcal{R}=\mathbb{F}_{q^2}(\tilde{x},\tilde{y}): \tilde{y}^2=\tilde{x}^q+\tilde{x}.$$
Since the automorphism group of $\mathcal{R}$ is known, we immediately find the automorphism group of $\cF_{m-1}$ as well.

\begin{lemma} \label{Roquette}
$\cF_{m-1}$ and $\mathcal{R}$ are isomorphic.
\end{lemma}
\begin{proof}
Consider the map
$$\begin{cases} \cF_{m-1} \rightarrow \mathcal{R}\\ (x,y) \mapsto \bigg(\frac{y+2x}{2(y-2x)},\frac{x^{m+1}-x^{m-1}}{(y-2x)^m} \bigg)=:(\tilde{x}, \tilde{y})\end{cases}.$$
Then from $y^{m}=x^{m-1}(x^2+1)$, we have $y^{q+1}=x^{q-1}(x^2+1)^2$. Now we prove $\tilde{y}^2=\tilde{x}^q+\tilde{x}$: \begin{align*}
\tx^q+\tx&= \left( \frac{y+2x}{2(y-2x)}\right)^q+ \left( \frac{y+2x}{2(y-2x)}\right)\\
&=\frac{(y+2x)^q(y-2x)+(y+2x)(y-2x)^q}{2(y-2x)^{q+1}}\\
&=\frac{2y^{q+1}-8x^{q+1}}{2(y-2x)^{q+1}}\\
&=\frac{2x^{q-1}(x^2+1)^2-8x^{q+1}}{2(y-2x)^{q+1}}\\
&=\frac{x^{q+3}-2x^{q+1}+x^{q-1}}{(y-2x)^{q+1}}\\
&=\ty^2.
\end{align*}
Since the function fields have the same genus $m-1$ and $m-1 >1$, the map above is an isomorphism.
\end{proof}

The Roquette function field has been studied extensively and in particular it is known that its automorphism group has cardinality $2 \cdot |\mathrm{PGL}(2,q)|$ and can be described as an extension of $\mathrm{PGL}(2,q)$ by a group of order two \cite[Theorem 12.11]{HKT2008}.

\bigskip

The following theorem summarizes our knowledge about $\aut(\cF_i)$, collecting all the partial results proven in this section. The statements for $q \le 31$ were obtained using a computer. 

\begin{theorem} \label{AUTOFINAL}
Let $q=p^n$ be the power of an odd prime $p$. Let $m=(q+1)/2$
and $\cF_i=\overline{\mathbb{F}}_{q^2}(x,y)$ with $y^m=x^i(x^2+1)$, $\gcd(i,m)=\gcd(i+2,m)=1$. Then $\aut(\cF_i)=G_i$ unless $(i,p)=(1,3)$, $(i,q)=(1,7)$ or $i \in \{(m-2)/2,m-1\}.$ Moreover:
\begin{itemize}
    \item $\aut(\cF_i)$ is infinite, if $(i,q)=(1,3)$,
    \item $\aut(\cF_i)$ is the semidirect product of $G_i$ and a cyclic group of order $3$ if $(i,p)=(1,3)$ and $q>3$,
    \item $\aut(\cF_i)$ is the semidirect product of a solvable group of order $48$ and a cyclic group of order $2$ if $(i,q)=(1,7)$,
    \item $\aut(\cF_i)$ contains a subgroup that is the semidirect product of $G_i$ and a cyclic group of order $4$ if $i=(m-2)/2$,
    \item $\cF_i$ is isomorphic to the Roquette function field and in particular $\aut(\cF_i)$ is an extension of $\mathrm{PGL}(2,q)$ by a group of order $2$ if $i=m-1.$
\end{itemize}
%
%
\end{theorem} 

\begin{corollary}\label{short_orbits_in_omega_cor}
Let $q=p^n$ be the power of an odd prime $p$. Let $m=(q+1)/2$ and $i=1,\ldots,(m-3)/2$ such that $\gcd(i,m)=\gcd(i+2,m)=1$. Unless $(i,q)=(1,3)$ or $(i,q)=(1,7)$, $Aut(\cF_i)$ acts on $\Omega=\{P_0, P_\infty, P_\alpha, P_{-\alpha}\}$. Moreover, any short orbit of $Aut(\cF_i)$ is contained in $\Omega$. 
\end{corollary}

\begin{proof}
  From Theorem \ref{AUTOFINAL}, unless $i=1$ and either $q$ is a power of $3$ or $q=7$, $\aut(\cF_i)=G_i$ and hence fixes $P_0$ and $P_\infty$ and has $\{P_\alpha, P_{-\alpha}\}$ as another short orbit. Note that $G_i$ has no other short orbits.
  
  The fact that $\aut(\cF_i)$ acts on $\Omega$ is true also if $p=3$, $i=1$ and $q>3$. In fact in this case Theorem \ref{AUTOFINAL} gives that $\aut(\cF_1)=C_3 \rtimes G_1$. This group fixes $P_\infty$ (as both $C_3$ and $G_1$ do so) and acts on $\Omega \setminus \{P_\infty\}$ transitively. Since neither $G_1$ nor $C_3$ have short orbits outside $\Omega$, the group $\aut(\cF_1)$ acts with long orbits elsewhere. 
\end{proof}

Theorem \ref{AUTOFINAL} gives a precise description of $\aut(\cF_i)$ except when $i=(m-2)/2$ and $q>7$. We believe that in this case the automorphisms described in Theorem \ref{aut_bound_special_case_thm} form the entire automorphism group. We finish this section by stating this as a conjecture.

\begin{conjecture}
Let $q>7$ and $i=(m-2)/2$. Then $|\aut(F_i)|=4(q+1).$
\end{conjecture}

\section{Isomorphism classes}

Using the results from Section \ref{sec_initial_results}, we will in this section complete the study of the isomorphism classes. Whereas previously, we used the notation $P_0,P_\a,P_{-\a}$ and $P_\infty$ to indicate certain places of $\cF_i$, we now change the notation to $P_0^i,P_{\a}^i,P_{-\a}^i$ and $P_\infty^i$. The reason is that we will have to keep track of places of several function fields at the same time. Similarly, to be able to describe maps between different function fields $\cF_i$ and $\cF_j$, we will in this section sometimes write $\cF_i= \mathbb{F}_{q^2}(x_i, y_i)$ with $y_i^m = x_i^i (x_i^2 + 1)$, when the role of $i$ needs to be emphasized.

As mentioned before, when studying isomorphism classes, we only need to consider the function fields $\cF_i$ for $1 \le i \le (m-2)/2$ and $i=m-1$. We start by showing that $\cF_{m-1}$ is not isomorphic to any of the others.

\begin{theorem}
For any $i\in \mathbb{Z}$, satisfying $1\leq i \leq (m-2)/2$ and $\gcd(i,m) = \gcd(i+2,m) = 1$, the function fields $\cF_i$ and $\cF_{m-1}$ are non-isomorphic.
\end{theorem}

\begin{proof}
From the size of the automorphism groups as listed in Theorem \ref{AUTOFINAL}, we see that $\cF_i$ and $\cF_{m-1}$ cannot be isomorphic except possibly when $i=(m-2)/2.$ The function field $\overline{\mathbb F}_{q^2}\cF_{m-1}$ has a place $P$ such that two is a gap of $P$, since $\cF_{m-1}$ is isomorphic to the Roquette function field. On the other hand, very similar as in the proof of Theorem \ref{actiononOmega}, one shows that $\overline{\mathbb F}_{q^2}\cF_{(m-2)/2}$ has no such place $Q$ using Remark \ref{rem:special_i} for $Q \in \Omega$ and the holomorphic differential $\frac{(x-a)x^{(m-2)/2-1}dx}{y^{m-1}}$, with $a:=x(Q)$, for $Q \not\in \Omega$. 
\end{proof}

We are left with studying whether or not for some $i, j \in \mathbb{Z}$, satisfying $1\leq i < j \leq (m-2)/2$, the function fields $\cF_i$ and $\cF_j$ can be isomorphic. We will see that this cannot happen.

\begin{theorem}\label{thm:non_isomorphic}
    For any $i, j \in \mathbb{Z}$, satisfying $1\leq i < j \leq \lfloor (m-2)/2 \rfloor$ and $\gcd(i,m) = \gcd(i+2,m) = \gcd(j,m) = \gcd(j+2, m) = 1$, the function fields $\cF_i$ and $\cF_j$ are non-isomorphic.
\end{theorem}

\begin{proof}
    First, note that $1 \leq i < j \leq (m-2)/2$ only occurs when $(m-2)/2 \geq 2$, so we can assume $q \geq 11$. We may conclude then, from Theorem \ref{AUTOFINAL}, that the order of $\aut(\cF_i)$ is at most $3(q+1)$, and for $j = \frac{m-2}{2}$ the statement of the theorem follows already from Theorem \ref{aut_bound_special_case_thm}.

    Now, assume that $1\leq i<j \leq \lfloor \frac{m-3}{2} \rfloor$, and assume for contradiction that we have an isomorphism $\phi : \cF_j \to \cF_i$. We denote the induced map from the set of places of $\cF_i$ to the set of places of $\cF_j$ by $\phi^*$. Such a map $\phi^*$ maps short orbits to short orbits, so it follows from Corollary \ref{short_orbits_in_omega_cor} that the places $P_0^i, P_\infty^i,$ and $P_{\pm\alpha}^i$ of $\cF_i$ must be mapped to the places $P_0^j, P_\infty^j,$ and $P_{\pm\alpha}^j$ of $\cF_j$. We will use the results on the semigroups at these points to show that such an isomorphism cannot exist. 
    
    The first step towards finding a contradiction is to show that $\phi^*(P_\infty^i) = P_\infty^j$. Just as in the proof of Theorem \ref{distinguishing_Palph_Pinf_cor}, we observe that $i+2 \in H(P_\infty^i)$ while $i+2 \notin H(P_{\pm\alpha}^j)$, so $H(P_\infty^i) \neq H(P_{\pm\alpha}^j)$. Similarly, $j+2 \in H(P_\infty^j)$ while $j+2 \notin H(P_{\pm\alpha}^i)$, so $H(P_\infty^j) \neq H(P_{\pm\alpha}^i)$. It follows that $P_\infty^i$ cannot be mapped to $P_{\pm\alpha}^j$ by $\phi^*$ and that $P_\infty^j$ cannot be mapped to $P_{\pm\alpha}^i$ by the inverse of $\phi^*$. There are two remaining options: Either $\phi^*$ maps $P_\infty^i$ to $P_\infty^j$ or it maps $P_\infty^i$ to $P_0^j$ and $P_0^i$ to $P_\infty^j$.

    Assume that the latter happens. Since either $\phi^*(P_{\pm\alpha}^i) = P_{\pm\alpha}^j$ or $\phi^*(P_{\pm\alpha}^i) = P_{\mp\alpha}^j$, we find a principal divisor
    $$
        \left( \frac{y_i}{\phi(y_j)} \right)_{\cF_i} = (i+j+2) (P_0^i - P_\infty^i).
    $$
    At the same time, we have the divisor $(x_i)_{\cF_i} = m(P_0^i - P_\infty^i)$, so it follows that $\gcd(i+j+2, m) \in H(P_\infty^i)$. Since $i+j+2 < m$, this is in contradiction with Corollary \ref{divisors_of_m_are_gaps_cor} and we conclude that $\phi^*$ maps $P_\infty^i$ to $P_\infty^j$.

    Next, we investigate whether $P_0^i$ may be mapped to $P_0^j$ by $\phi^*$. If $\phi^*(P_0^i) = P_0^j$ then we find a principal divisor
    $$
        \left( \frac{\phi(y_j)}{y_i}\right)_{\cF_i} = (j-i) (P_0^i - P_\infty^i),
    $$
    and we obtain a contradiction by a similar argument as above. 

    The only remaining option is $\phi^*(P_0^i) = P_{\pm\alpha}^j$. For $i > 2$ and $m\geq 19$ this is excluded by Lemma \ref{P0_polenumber_lessthan_lemma}. Similarly, by considering the image of $P_0^j$ under the inverse of $\phi^*$, we can, for $m\geq 19$, exclude all cases except $j = \lfloor \frac{m-3}{2} \rfloor$. For $i=2$ and $j = \lfloor \frac{m-3}{2} \rfloor$ it follows from $\gcd(i,m) = 1$ that $m$ is odd, so in fact $j = \frac{m-3}{2}$, and we reach a contradiction by using Lemma \ref{Palph_P0_specialcase_lemma}.

    Finally, we use Magma to calculate the relevant semigroups for $m<19$, and conclude that an isomorphism between $\cF_i$ and $\cF_j$ cannot exist. 
\end{proof}

What is left to do is to obtain an expression for the total number of distinct isomorphism classes of the function fields $\cF_i$. This number can be determined from the prime factorization of $m$. 

\begin{theorem} \label{numbernoniso}
    Write $m = 2^{\alpha_0} p_1^{\alpha_1} p_2^{\alpha_2} \cdots p_n^{\alpha_n}$, for distinct odd primes $p_1, \dots, p_n$, with $\alpha_0 \in \mathbb{Z}_{\geq 0}$ and $\alpha_1, \dots, \alpha_n \in \mathbb{Z}_{>0}$. The number of isomorphism classes among the function fields $\cF_i$ defined in Equation \eqref{curveC} is
    $$
        N(m) = 
        \begin{cases}
            \frac{\phi_2(m) + 1}{2} & \text{ for } m \not\equiv 0 \pmod 4,\\
            \frac{\phi_2(m) + 2}{2} & \text{ for } m \equiv 0 \pmod 4,
        \end{cases}
    $$
    where $\phi_2(m)$ is given by 
    $$
        \phi_2(m) = 2^{\max\{0,\alpha_0-1\}} p_1^{\alpha_1 - 1}(p_1-2) \cdots p_n^{\alpha_n -1}(p_n - 2).
    $$
\end{theorem}

\begin{proof}
    We immediately reduce the problem to counting the distinct isomorphism classes among the set of function fields corresponding to $i \in \{0, 1, \dots, m-1\}$, for which $\gcd(i,m) = \gcd(i+2,m) = 1$. We claim that there are exactly $\phi_2(m)$ function fields in this set, i.e., that
    $$
        \phi_2(m) = |\{ i \in \{0, 1, \dots, m-1\} \mid \gcd(i,m) = \gcd(i+2,m) = 1 \}|.
    $$
    
    Assume first that $m$ is odd, so that $\alpha_0 = 0$. By the Chinese Remainder Theorem, each $i \in \{1, 2, \dots, m-1\}$ corresponds exactly to a tuple $( x_1, x_2, \dots, x_n)\in \mathbb{Z}^n$, with $0\leq x_j < p_j^{\alpha_j}$ and $i \equiv x_j \pmod{p_j^{\alpha_j}}$ for $j = 1, \dots, n$. There is a non-trivial common divisor of $i$ and $m$ if and only if $x_j$ is divisible by $p_j$ for at least one $j$. Similarly, $i+2$ and $m$ have a non-trivial common divisor if and only if $x_j + 2$ is divisible by $p_j$ for at least one $j$. 
    
    For each $j \in \{1,\dots, n\}$, there are $p_j^{\alpha_j-1}(p_j-2)$ possible choices for $x_j$ such that neither $x_j$ nor $x_j + 2$ is divisible by $p_j$. It follows that there are exactly $\phi_2(m)$ tuples, and hence $\phi_2(m)$ possible values for $i$, such that the conditions on the greatest common divisors are satisfied.

    For even $m$, i.e., for $\alpha_0 \geq 1$, a very similar argument applies. The only thing that needs to be considered when dealing with the prime 2, as opposed to the odd primes, is that $i \equiv i + 2 \pmod{2}$. The number of possible choices for $i$ modulo $2^{\alpha_0}$ is then $2^{\alpha_0-1}(2-1) = 2^{\alpha_0-1}$. This finishes the proof of the claim.

    Now, among these $\phi_2(m)$ different possible values for $i$, we know that $i = m-1$ corresponds to a function field that is not isomorphic to any of the others, namely the Roquette function field. It is clear that $m-2$ is never among the possible values of $i$. For $1 \leq i \leq m-3$ we know that $\cF_i$ is isomorphic to $\cF_{m-2-i}$. Note that $i$ satisfies the greatest common divisor conditions if and only if $m-2-i$ satisfies the conditions. This means that the remaining $\phi_2(m) - 1$ function fields split into pairs of isomorphic function fields, with the possible exception of the function field $\cF_{(m-2)/2}$. As previously mentioned, $(m-2)/2$ is among the possible values of $i$ precisely when $m \equiv 0 \pmod{4}$. 
    
    Each pair contains a function field $\cF_i$ with index $i$ less than or equal to $\frac{m-2}{2}$, so by Theorem \ref{thm:non_isomorphic} function fields from different pairs are non-isomorphic. For $m \not\equiv 0 \pmod{4}$ this means that we get 
    $$
        1 + \frac{\phi_2(m)-1}{2} =  \frac{\phi_2(m)+1}{2}
    $$ 
    isomorphism classes, and for $m \equiv 0 \pmod{4}$ we get
    $$
        2 + \frac{\phi_2(m) - 2}{2} = \frac{\phi_2(m) + 2}{2}
    $$
    isomorphism classes, as claimed.
\end{proof}

\section*{Acknowledgments}
This work was supported by a research grant (VIL”52303”) from Villum Fonden.

\bibliographystyle{abbrv}

\bibliography{bibmanypoints} 

\end{document}